\documentclass{aic}

\aicAUTHORdetails{%
  title = {Ramsey Goodness of Books Revisited}, 
  author = {Jacob Fox, Xiaoyu He, and Yuval Wigderson},
  plaintextauthor = {Jacob Fox, Xiaoyu He, Yuval Wigderson},
  keywords = {Ramsey theory, Ramsey goodness, book graphs},
}   

\usepackage{amssymb, amsmath, amsthm, graphicx}
\usepackage[left=1in,top=1in,right=1in]{geometry}
\usepackage{todonotes}
\usepackage{comment,enumitem,mathtools}
\usepackage[capitalize]{cleveref}
\usepackage{color}

\theoremstyle{plain}
      \newtheorem{theorem}{Theorem}[section]
      \newtheorem{lemma}[theorem]{Lemma}
            \newtheorem{claim}[theorem]{Claim}

      \newtheorem{conjecture}[theorem]{Conjecture}
      \newtheorem*{thm:main1}{Theorem \ref{thm:main1}}
      \newtheorem*{thm:multipartite}{Theorem \ref{thm:multipartite}}
      \newtheorem*{thm:removal-above-threshold}{Theorem \ref{thm:removal-above-threshold}}
\theoremstyle{definition}
      \newtheorem{definition}[theorem]{Definition}

\theoremstyle{remark}

\newcommand\ab[1]{\lvert#1\rvert}
\newcommand\up[1]{^{(#1)}}
\newcommand\ol[1]{\overline{#1}}
\newcommand\eps{\varepsilon}

\DeclareMathOperator{\pr}{Pr}
\newcommand\E{\mathbb{E}}

\DeclareMathOperator{\ext}{ext}

\newcommand\book[2]{B_{#1, #2}}

\aicEDITORdetails{%
   year={2023},
   number={4},
   received={27 September 2021},   
   published={29 July 2023},  
   doi={10.19086/aic.2023.4},      
}   

\begin{document}

\begin{frontmatter}[classification=text]


\author[jfox]{Jacob Fox\thanks{Department of Mathematics, Stanford University, Stanford, CA 94305, USA. Email: \url{jacobfox@stanford.edu}. Research supported by a Packard Fellowship and by NSF grants DMS-1953990 and DMS-2154129.}}
\author[xhe]{Xiaoyu He\thanks{Department of Mathematics, Princeton University, Princeton, NJ 08544, USA. Email: \url{xiaoyuh@princeton.edu}. Research supported by the NSF Mathematical Sciences Postdoctoral Research Fellowships Program under Grant DMS-2103154.}}
\author[ywig]{Yuval Wigderson\thanks{School of Mathematics, Tel Aviv University, Tel Aviv 69978, Israel. Email: \url{yuvalwig@tauex.tau.ac.il}. Research supported by NSF GRFP Grant DGE-1656518, NSF-BSF Grant 20196, and by ERC Consolidator Grant 863438 and ERC Starting Grant 101044123.}}

\begin{abstract}
The Ramsey number $r(G,H)$ is the minimum $N$ such that every graph on $N$ vertices contains $G$ as a subgraph or its complement contains $H$ as a subgraph.
For integers $n \geq k \geq 1$, the $k$-book $\book kn$ is the graph on $n$ vertices consisting of a copy of $K_k$, called the \emph{spine}, as well as $n-k$ additional vertices each adjacent to every vertex of the spine and non-adjacent to each other.
A connected graph $H$ on $n$ vertices is called {\it $p$-good} if $r(K_p,H)=(p-1)(n-1)+1$. Nikiforov and Rousseau proved that if $n$ is sufficiently large in terms of $p$ and $k$, then $\book kn$ is $p$-good. Their proof uses Szemer\'edi's regularity lemma and gives a tower-type bound on $n$. We give a short new proof that avoids using the regularity method and shows that every $\book kn$ with $n \geq 2^{k^{10p}}$ is $p$-good. 

Using Szemer\'edi's regularity lemma, Nikiforov and Rousseau also proved much more general goodness-type results, proving a tight bound on $r(G,H)$ for several families of sparse graphs $G$ and $H$ as long as $|V(G)| < \delta |V(H)|$ for a small constant $\delta > 0$. Using our techniques, we prove a new result of this type, showing that $r(G,H) = (p-1)(n-1)+1$ when $H =\book kn$ and $G$ is a complete $p$-partite graph whose first $p-1$ parts have constant size and whose last part has size $\delta n$, for some small constant $\delta>0$. Again, our proof does not use the regularity method, and thus yields double-exponential bounds on $\delta$.
\end{abstract}
\end{frontmatter}


\section{Introduction}
For two graphs $G,H$, their \emph{Ramsey number} $r(G,H)$ is the smallest $N$ such that every graph $\Gamma$ on $N$ vertices contains $G$ as a subgraph, or its complement contains $H$ as a subgraph. The existence of $r(G,H)$ is guaranteed by Ramsey's theorem \cite{Ramsey}. The most well-studied Ramsey number is the \emph{diagonal} Ramsey number $r(K_k, K_k)$. One of the oldest (and easiest) results in Ramsey theory is the fact that $r(K_k,K_k) \geq (k-1)^2+1$ (this is implicit already in the work of Erd\H{o}s--Szekeres~\cite{ErSz}), which is proved by taking $\Gamma$ to be the complete balanced $(k-1)$-partite graph on $(k-1)^2$ vertices.

This quadratic lower bound is far from best possible. Indeed, it is known \cite{Erdos47,ErSz} that $r(K_k,K_k)$ must grow exponentially in $k$, though the exact exponential rate remains unknown despite decades of intense research. Nonetheless, it is an instance of a much more general inequality which can be tight. Write $\chi(G)$ for the chromatic number of $G$, and $s(G)$ for the minimum size of a color class in any proper coloring of $G$ with $\chi(G)$ colors. The inequality in question is then
\begin{equation}\label{eq:goodness}
r(G,H) \ge (\chi(G)-1)(\ab{V(H)}-1)+s(G),
\end{equation}
which holds under the assumption that $H$ is a connected graph with at least $s(G)$ vertices. Inequality (\ref{eq:goodness}) was first proved by Burr~\cite{Burr}, by taking $\Gamma$ to be a complete $\chi(G)$-partite graph with $\chi(G)-1$ parts of size $\ab{V(H)}-1$ and one part of size $s(G)-1$. 

Burr and Erd\H os \cite{BuEr} initiated the study of when (\ref{eq:goodness}) is tight; following their terminology, one says that a connected graph $H$ is \emph{$G$-good} if (\ref{eq:goodness}) is tight. In case $G=K_p$, one says that $H$ is \textit{$p$-good}, rather than $K_p$-good.

While the Ramsey goodness bound (\ref{eq:goodness}) is far from tight in the case of cliques, it turns out that many interesting graphs are $p$-good, and that the theory of Ramsey goodness generalizes many important results in graph theory. For example, Tur\'an's theorem, which states that the balanced complete $(p-1)$-partite graph has the most edges among all $K_p$-free graphs on $N$ vertices, is equivalent to the fact that stars are $p$-good. 
Extending this fact, Chv\'atal~\cite{Chvatal} proved that all trees are $p$-good for all $p \geq 3$, and this theorem inspired Burr and Erd\H os to define Ramsey goodness. At this point, there is a rich theory of Ramsey goodness, about which we refer the interested reader to the survey of Conlon, Sudakov, and the first author \cite[Section 2.5]{CoFoSu15}.

For $n \geq k \geq 1$, the $k$-book $\book kn$ on $n$ vertices consists of a copy of $K_k$, called the \emph{spine}, as well as $n-k$ additional vertices each joined to every vertex of the spine; equivalently, $\book kn$ consists\footnote{We remark that other notation exists for book graphs; notably, some other papers (e.g.\ \cite{Conlon,CFW,NR04}) use $B_{n-k}\up k$ to denote what is $\book kn$ in our notation.} of $n-k$ cliques of order $k+1$ sharing a common $K_k$. Book graphs arise naturally in the study of Ramsey numbers. Indeed, Ramsey \cite{Ramsey} originally proved the finiteness of $r(K_k,K_k)$ by proving the finiteness of $r(\book kn, \book kn)$ for every $n$, and it was observed by Erd\H os, Faudree, Rousseau, and Schelp \cite{ErFaRoSc} that the classical Erd\H os--Szekeres \cite{ErSz} upper bound on Ramsey numbers can also be framed as an upper bound on certain book Ramsey numbers. This connection yields an important approach to improving upper bounds on $r(K_k,K_k)$; for more details, see e.g.\ \cite{Conlon,CFW}.

In \cite{NR04}, Nikiforov and Rousseau used Szemer\'edi's regularity lemma to prove that for every $k,p \geq 2$ and every sufficiently large $n$, the book $\book kn$ is $p$-good. One consequence of applying the regularity method is that their proof yields tower-type bounds on how large $n$ must be in terms of $k$ and $p$, and they raised the question of what the best possible $n$ is. Our first main result is a new proof of $p$-goodness for books which avoids the use of the regularity lemma, and thus gets a much better dependence for $n$ on $k$ and $p$.
\begin{theorem}\label{thm:main1}
For all $k,p \geq 2$, if $n \geq 2^{k^{10p}}$, then $\book kn$ is $p$-good; that is, $$r(K_p,\book kn) = (p-1)(n-1)+1.$$
\end{theorem}
Our main technique is a novel variant of the greedy embedding strategy, which allows us to build a large induced copy of a complete multipartite graph inside a $K_p$-free graph whose complement does not contain a very large book. We do not expect the bound $n \geq 2^{k^{10p}}$ to be best possible, and we discuss this further in the concluding remarks.

Extending the techniques from \cite{NR04}, Nikiforov and Rousseau \cite{NR09} were later able to prove a remarkable theorem, which remains the most general result in the field of Ramsey goodness. It immediately implies that many families of graphs are $p$-good, such as clique subdivisions and sufficiently large planar graphs (see \cite[Theorem 1.2]{MR3480560}). As the result in its full generality requires some notation, we state only the following special case, which applies only to book graphs.

\begin{theorem}[{Nikiforov and Rousseau \cite[Theorem 2.12]{NR09}}] \label{thm:book-book}
    For every $k,p\geq 2$, there exists some $\delta>0$ such that for all sufficiently large $n$,
    \[
        r(\book {p-1}{\delta n}, \book kn) = (p-1)(n-1)+1.
    \]
\end{theorem}
In other words, the Ramsey goodness result $r(K_p, \book kn) = (p-1)(n-1)+1$ remains true even if we replace $K_p$ by the much larger graph $\book {p-1}{\delta n}$ containing it. 
\cref{thm:book-book} thus goes beyond the basic Ramsey goodness framework introduced by Burr and Erd\H os, because it shows that $r(G,H) = (p-1)(n-1)+1$ even in some cases when $G$ is not a fixed graph. For more on \cref{thm:book-book}, as well as on what happens when the two books have roughly comparable numbers of vertices, we refer the reader to \cite{CFW2} (as well as the earlier papers \cite{Conlon, CFW}).

Just as before, the proof of \cite{NR09} uses Szemer\'edi's regularity lemma, and hence the bound on $1/\delta$ in \cref{thm:book-book} is of tower type. In order to demonstrate the flexibility of our proof technique, we prove the following generalization of \cref{thm:book-book}, which again goes beyond the basic Ramsey goodness framework of Burr and Erd\H os, as both $G$ and $H$ are allowed to grow with $n$. It generalizes \cref{thm:book-book} because the book $\book{p-1}{\delta n}$ is a complete $p$-partite graph in which all but one part consist of a single vertex.
\begin{theorem}\label{thm:multipartite}
    For every $k,p,t \geq 2$, there exists $\delta>0$ such that the following holds for all $n\ge 1$. Fix positive integers $1 \leq a_1 \leq \dotsb \leq a_{p-1} \leq t$ and $a_p \leq \delta n$. Let $G$ be the complete $p$-partite graph with parts of sizes $a_1,\dots,a_p$. Then $r(G,\book kn) = (p-1)(n-1)+a_1$ if and only if $a_1 = a_2 = 1$.
    
    Additionally, we may take $\delta\geq 2^{-t^{1000k^2p^2}}$.
\end{theorem}

Although \cref{thm:multipartite} has not appeared in the literature, the ``if'' direction (which is the harder one) can be deduced from the general theorem of Nikiforov and Rousseau \cite[Theorem 2.1]{NR09}. Nonetheless, the main novelty is not the statement of \cref{thm:multipartite}, but rather the fact that our proof again avoids the use of the regularity lemma, so that the bounds on $1/\delta$ are not of tower-type. It would be very interesting to see how far one can push these ideas; for example, is it possible to completely eliminate the use of the regularity lemma from the proof of \cite[Theorem 2.1]{NR09}?

\paragraph{Organization.} In \cref{sec:good-book}, we warm up by proving \cref{thm:main1}; along the way, we prove some general lemmas that set the groundwork for \cref{thm:multipartite}. In \cref{sec:main-proof}, we prove Theorem~\ref{thm:multipartite}. Part (ii) of the theorem is a short explicit construction, but part (i) requires a variant of the Andr\'asfai--Erd\H os--S\'os theorem, \cref{thm:aes-variant}, which we prove in Sections~\ref{sec:stability-supersaturation} and~\ref{sec:aes-variant}. \cref{thm:aes-variant} is an important ingredient in the proof of \cref{thm:multipartite}, as it essentially allows us to reduce to the case that such a $G$-free graph is $(p-1)$-partite. While such a statement is relatively standard, the specific version we need is apparently new. Finally, we collect some interesting open problems in \cref{sec:concluding}.

\paragraph{Notation and terminology.} For positive integers $p,a_1,\dots,a_p$, let $K_p(a_1,\dots,a_p)$ denote the complete $p$-partite graph with parts of sizes $a_1,\dots,a_p$. In case $a_1=\dotsb=a_p=a$, we denote this by $K_p(a)$. We denote the number of edges in a graph $\Gamma$ by $e(\Gamma)$, and the number of edges between vertex subsets $X$ and $Y$ by $e(X,Y)$. We say that a graph $\Gamma$ contains a \emph{copy} of a graph $G$ if $\Gamma$ has a (not necessarily induced) subgraph isomorphic to $G$. If $\Gamma$ has no copy of $G$, we say that $\Gamma$ is \emph{$G$-free}. For positive real numbers $x$ and $y$, we denote by $x \pm y$ any quantity in the interval $[x-y,x+y]$. All logarithms are to base 2 unless otherwise noted.
For the sake of clarity of presentation,
we omit floor and ceiling signs when they are not crucial.

\section{Ramsey goodness of books}\label{sec:good-book}

In this section, we prove \cref{thm:main1} using a greedy embedding strategy, which we first describe informally. Assume for the sake of contradiction that there exists graph $\Gamma$ on $(p-1)(n-1)+1$ vertices which satisfies:
\begin{enumerate}
\item \label{property:kp}$\Gamma$ is $K_p$-free.
\item \label{property:book}$\overline{\Gamma}$ contains no copy of $B_{k,n}$.
\end{enumerate}
We iteratively find smaller and smaller induced complete multipartite subgraphs $H_1, \ldots, H_{p-1}, H_p$ of $\Gamma$ by applying these properties, where $H_i=K_i(t_i)$ is an induced complete $i$-partite subgraph. 

An induced copy of $H_1 = K_1(t_1)$ is just an independent set of size $t_1$,  which exists by applying Property \ref{property:kp} and Ramsey's theorem. To find $H_2$, which is an induced complete bipartite subgraph, we use Property \ref{property:book} to see that a substantial portion of the vertices of $\Gamma$ have many edges to $H_1$, so we can find many vertices $U$ with the same common neighborhood in $H_1$. Within $U$, we find a large independent set $U'$ via Ramsey's theorem and Property \ref{property:kp}, and $U' \cup (N(U')\cap H_1)$ forms the desired $H_2$. Continuing in this fashion, we build each $H_i$ out of the previous $H_{i-1}$ by using Property \ref{property:book} to find many common neighbors to extend to, and Property \ref{property:kp} with Ramsey's theorem to find an independent set among those common neighbors. The existence of the final $H_p = K_p$ in $\Gamma$ provides the desired contradiction.

The following result is the greedy embedding lemma that we use. Given a graph $\Gamma$, it allows us to find a large book in $\ol \Gamma$ or find a large induced complete multipartite subgraph of $\Gamma$.

\begin{lemma}\label{lem:greedy}
Let $k, r, s, t$ be positive integers with $s \leq t$ and $2k \le t$, and let $G$ be any graph.  Let $\Gamma$ be a $G$-free graph with $N \geq \binom ts^r \frac{t}{2ks}r(G, K_s)$ vertices which contains $K_{r}(t)$ as an induced subgraph, with parts $V_1,\dots,V_r$. If $\overline \Gamma$ does not contain a book $\book kn$ with $n \geq (1-4ks/t)N/r$ vertices, then $\Gamma$ contains an induced copy of $K_{r+1}(s)$ with parts $W_0,\dots,W_r$, where $W_i \subseteq V_i$ for every $1 \leq i \leq r$.
\end{lemma}
\begin{proof}
Let $\eps = s/t$. Partition the vertex set of $\Gamma$ into $r+1$ parts $U_0,U_1,\ldots,U_r$, where, for each $i \in [r]$, every vertex in $U_i$ has degree at most $\eps t$ to $V_i$, and every vertex in $U_0$ has degree at least $\eps t$ to each $V_j$. Note that by construction, $V_i \subseteq U_i$ for $i\in [r]$. 

Suppose there is $i \in [r]$ such that $|U_i| \geq (1-2k\eps)N/r$. Let $X$ denote the set of all vertices $v \in V_i$ with at most $2\eps|U_i \setminus V_i|$ neighbors in $U_i \setminus V_i$. Since each vertex in $U_i$ has density at most $\eps$ to $V_i$, we have $|X| \geq |V_i|/2 = t/2\geq k$. Let $Q$ be any $k$ vertices in $X$. Then all but at most a $2k\eps $ fraction of the vertices in $U_i \setminus V_i$ are empty to $Q$. So $Q$ together with the vertices of $U_i$ that have have no neighbors in $Q$ form a $k$-book in $\overline \Gamma$ with at least 
$$(1-2k\eps)|U_i \setminus V_i|+|V_i| = (1-2k\eps)(\ab{U_i} - \ab{V_i}) +\ab{V_i}\geq (1-2k\eps)\ab{U_i} \geq (1-4k\eps)N/r$$ vertices.

So we may assume that there is no $i \in [r]$ with $|U_i| \geq (1-2k\eps)N/r$. In this case, we have $|U_0| \geq N-r(1-2k\eps)N/r=2k\eps N$. By the pigeonhole principle, there is 
a subset $T \subset U_0$ of size at least ${\binom ts}^{-r}|U_0| \geq r(G, K_s)$ such that there exist subsets $W_i \subseteq V_i$ with $|W_i|=s$ for $i\geq 1$ such that every vertex in $T$ is complete to each $W_i$. As $\Gamma$ and hence the induced subgraph $\Gamma[T]$ is $G$-free and $|T| \geq r(G, K_s)$, we know that $T$ contains an independent set $W_0$ of order $s$. Then $W_0,W_1,\ldots,W_r$ form a complete induced $(r+1)$-partite subgraph of $\Gamma$ with parts of size $s$. 
\end{proof}
Our next lemma shows that, once we find a large induced complete multipartite subgraph of $\Gamma$, we can find a large book in $\ol \Gamma$.

\begin{lemma}\label{lem:simple}
If a $K_p$-free graph $\Gamma$ on $n$ vertices contains $K_{p-1}(k)$ as an induced subgraph, then its vertex set can be partitioned into $p-1$ subsets that each span a $k$-book in $\overline \Gamma$. 
\end{lemma}
\begin{proof}
Let $V_1,\ldots,V_{p-1}$ be the $p-1$ parts of the induced $K_{p-1}(k)$. As $\Gamma$ is $K_p$-free, each vertex in $\Gamma$ has no neighbors in some $V_i$. Partition the vertex set of $\Gamma$ into $p-1$ parts $U_1,\ldots,U_{p-1}$, where, for each $i \in [p-1]$, each vertex in $U_i$ has no neighbors in $V_i$. Then each $U_i$ spans a $k$-book in $\overline \Gamma$ with spine $V_i$. 
\end{proof}

Our next result is the main form in which we use \cref{lem:greedy}, and follows from it by a simple inductive argument.

\begin{lemma}
\label{lem:find-blowup}
Let $k,p,x,n$ be positive integers, and let $z = x\cdot(20k)^{p}$. Let $G$ be any graph. Let $\Gamma$ be a $G$-free graph on at least $N=(p-1)(n-1)+1$ vertices, and suppose $S \subseteq V(\Gamma)$ satisfies $\ab S \ge z^z \cdot r(G, K_z)$. Then either $\ol \Gamma$ contains a copy of $\book kn$, or else $\Gamma$ contains $K_{p-1}(x)$ as an induced subgraph, one part of which is a subset of $S$.
\end{lemma}

\begin{proof}
For $r=1,\ldots,p-2$, let $\eps_r = \left(1-r/(p-1)\right)/(4k)$ so that $(1-4k\eps_r)/r=1/(p-1)$. Let $t_{p-1}=x$ and $t_r=t_{r+1}/\eps_r$ for $r=p-2,\ldots,1$. 
Observe that 
\[
t_1=t_{p-1}/\prod_{r=1}^{p-2} \eps_r = x(4k)^{p-2}(p-1)^{p-2}/(p-2)!<(20k)^p x = z.
\]
Since $t_1 \geq t_2 \geq \dotsb \geq t_{p-1}$, this implies that $t_r <z$ for all $r$. We now prove by induction on $r$ for $r \in [p-1]$ that $\Gamma$ contains $K_r(t_r)$ as an induced subgraph, with the first part of $K_r(t_r)$ being a subset of $S$. 

For the base case $r=1$, we have $\ab S \geq r(G, K_z) \geq r(G,K_{t_1})$, so $\Gamma$ contains an independent set of order $t_1$, that is, $\Gamma[S]$ contains $K_r(t_r)$ with $r=1$ as an induced subgraph. 

Now suppose $\Gamma$ contains $K_r(t_r)$ as an induced subgraph, with the first part a subset of $S$. We apply Lemma \ref{lem:greedy} with $s=t_{r+1}$ and $t=t_r$.
Observe that 
\begin{eqnarray*}
\binom{t_{r}}{t_{r+1}}^r \left(\frac{t_r} {2kt_{r+1}}\right)r(G, K_{t_{r+1}})  & \leq & \left(\frac e{\eps_r}\right)^{rt_r}\left(\frac{t_r} {2kt_{r+1}}\right)r(G, K_{t_{r+1}}) \\ 
& < & z^z \cdot r(G, K_z) \le |S|.
\end{eqnarray*}
So either $\ol \Gamma$ contains a $k$-book with at least $(1-4k\eps_r)N/r=N/(p-1) > n-1$ vertices, in which case we are done, or $\Gamma$ contains an induced $K_{r+1}(t_{r+1})$ whose first $r$ parts are subsets of the $r$ parts of the $K_r(t_r)$. In particular, the first part of this induced $K_{r+1}(t_{r+1})$ is a subset of $S$. This proves the claimed inductive statement. The desired statement is just then the case $r=p-1$.
\end{proof}
We are now ready to prove \cref{thm:main1}, whose statement we now recall.
\begin{thm:main1}
For all $k,p \geq 2$, if $n \geq 2^{k^{10p}}$, then $\book kn$ is $p$-good; that is, $$r(K_p,\book kn) = (p-1)(n-1)+1.$$
\end{thm:main1}
\begin{proof}
Let $N=(p-1)(n-1)+1$ and $z=k(20k)^p$. Note that $z \leq (20k^2)^p \leq k^{7p}$ since $k \geq 2$, so
\[
    2z\log z \leq 2\cdot k(20k)^p \cdot 7p \log k = (\sqrt{14 k \log k})^2 (20k)^p p \leq (80 k^2)^p p \leq (160k^2)^p,
\]
since $p \geq 2$ and $\sqrt{k \log k} \leq k$ for all $k \geq 2$, and since $p \leq 2^p$. Finally, we observe that $160k^2 \leq k^{10}$ for all $k \geq 2$. The Erd\H os--Szekeres bound \cite{ErSz} implies that $r(K_p,K_z) \leq \binom{z+p}{p} \leq z^z$, and therefore,
\[
    z^z \cdot r(K_p,K_z) \leq z^{2z} = 2^{2z \log z} \leq 2^{k^{10p}} \leq n \leq N.
\]
Suppose for the sake of contradiction that there is a $K_p$-free graph on $N$ vertices such that $\overline \Gamma$ does not contain a $k$-book with $n$ vertices. By \cref{lem:find-blowup}, applied with $S=V(\Gamma)$ and $x = k$, we see that $\Gamma$ must contain $K_{p-1}(k)$ as an induced subgraph. But then Lemma \ref{lem:simple} implies that $\ol \Gamma$ contains a $k$-book with $n$ vertices as a subgraph, completing the proof. 
\end{proof}

\section{Proof of Theorem \ref{thm:multipartite}} \label{sec:main-proof}


In this section, we prove \cref{thm:multipartite}, which we now restate.
\begin{thm:multipartite}
For any $k,p,t \geq 2$, there exists $\delta\geq 2^{-t^{1000k^2p^2}}$ such that the following holds for all $n \geq 1$. Fix positive integers $1 \leq a_1 \leq \dotsb \leq a_{p-1} \leq t$ and $a_p \leq \delta n$. Let $G=K_p(a_1,a_2,\dots,a_p)$ and $H = \book kn$.
\begin{enumerate}[label=(\roman*)]
    \item\label{item:if} If $a_1=a_2=1$, then $r(G,H) = (p-1)(n-1)+a_1$.
    \item\label{item:only-if} If $a_2 \geq 2$, then $r(G,H) > (p-1)(n-1)+a_1$.
\end{enumerate}
\end{thm:multipartite}


We start with the construction for the ``only if'' direction, part (ii) of the theorem.

\begin{proof}[Proof of \cref{thm:multipartite}\ref{item:only-if}.]
Let $\Gamma$ be a graph on $N = (p-1)(n-1)+a_1$ vertices which are divided into $p-1$ parts $U_1,\ldots, U_{p-1}$ with $|U_1| = n+a_1-1$ and $|U_2| = \cdots = |U_{p-1}| = n-1$. The edges of $\Gamma$ are defined as follows. First, all pairs of vertices in two different parts are adjacent. Second, $U_1$ induces a $C_4$-free subgraph $A = \Gamma[U_1]$ which is almost $a_1$-regular. This means that either $A$ is $a_1$-regular (if $|U_1|$ or $a_1$ is even), or else all but one vertices of $A$ have degree $a_1$ and one vertex has degree $a_1-1$ (if $|U_1|$ and $a_1$ are both odd). Such a graph $A$ always exists if $n$ is large enough in terms of $a_1$; for example, a random graph with this degree sequence is $C_4$-free with positive probability for sufficiently large $n$ \cite[Corollaries 1--2]{Wormald}.

It remains to show that $\Gamma$ is $G$-free and $\overline{\Gamma}$ is $H$-free. Suppose $\Gamma$ contains a copy of $G$. Since $G$ is complete $p$-partite and $U_2,\ldots, U_{p-1}$ are independent sets of $\Gamma$, each of these sets can contain only vertices from at most one part of this copy of $G$. Thus, at least two parts of $G$ must be entirely contained inside $U_1$, which means that $\Gamma[U_1]$ must contain a copy of the complete bipartite graph $K_{a_1, a_2}$. By construction, $A = \Gamma[U_1]$ is $C_4$-free, so this is impossible unless $a_1=1$. When $a_1=1$, $A$ has maximum degree $1$ and thus cannot contain a copy of $K_{a_1, a_2}$ since $a_2 \ge 2$. In all cases, $\Gamma$ is $G$-free.

The complement $\overline{\Gamma}$ is a disjoint union of $\overline{A}$ and $p-2$ copies of $K_{n-1}$. The book $H$ is connected and has $n$ vertices, so $K_{n-1}$ cannot contain a copy of $H$. Also, $k\ge 2$, so $H$ contains at least two vertices of degree $n-1$, whereas $\overline{A}$ has either one or zero vertices of degree at least $n-1$. It follows that $\overline{A}$ contains no copies of $H$ either, completing the proof.
\end{proof}

The proof of \cref{thm:multipartite}\ref{item:if} divides into three parts. We first prove a stability-supersaturation result, which says that a graph with few copies of $K_p$ and with high minimum degree is close to $(p-1)$-partite. Using this, we prove the following 
variant of the Andr\'asfai--Erd\H{o}s--S\'os theorem,  \cref{thm:aes-variant} below, which states that a graph with high minimum degree and no copy of $K_p(a_1,\dots,a_p)$ is $(p-1)$-partite.

\begin{theorem}\label{thm:aes-variant}
For every $p,t \geq 2$ and every $1 =a_1 = a_2 \leq a_3 \leq \dotsb \leq a_{p-1} \leq t$, there exist some $\alpha,\delta>0$ such that if $m$ is large enough in terms of $t$ and $p$, $a_p \leq \delta m$, and $\Gamma$ is a $K_p(a_1,a_2,\dots,a_p)$-free graph on $m$ vertices with minimum degree at least $(1-1/(p-1)-\alpha)m$, then $\Gamma$ is $(p-1)$-partite.

Additionally, we may take $\alpha \geq 1/(200p^6 t^2), \delta \geq 2^{-t^{100p}}$, and the result holds for $m \geq tp^{20p}$.
\end{theorem}
Finally, we prove \cref{lem:low-deg-vtxs}, which states that under the assumptions of \cref{thm:multipartite}, almost all vertices of $\Gamma$ have high degree, meaning that we can apply \cref{thm:aes-variant} to conclude that most of $\Gamma$ is $(p-1)$-partite.

To conclude the proof, we use a careful averaging argument to show that under these assumptions, $\ol \Gamma$ must contain a copy of $H=\book kn$, completing the proof. In the next two subsections, we prove the stability-supersaturation theorem and \cref{thm:aes-variant}. In \cref{sec:proof-multipartite}, we prove \cref{lem:low-deg-vtxs} and complete the proof of \cref{thm:multipartite}\ref{item:if}.

\subsection{A stability-supersaturation theorem}\label{sec:stability-supersaturation}

In this section we prove one of the main ingredients of \cref{thm:aes-variant}, a variant of the Erd\H os--Simonovits stability version of Tur\'an's theorem. Roughly speaking, this result combines two types of well-known variants of Tur\'an's theorem. The first, namely the Erd\H os--Simonovits stability theorem \cite{MR0232703, MR0233735}, says that if $\Gamma$ is a $K_p$-free graph with slightly fewer edges than the Tur\'an graph, then $\Gamma$ can be turned into the Tur\'an graph by changing a small number of edges. The second, often known as a supersaturation result \cite{MR726456}, says that if $\Gamma$ is an $m$-vertex graph with slightly \emph{more} edges than the $K_p$-free Tur\'an graph, then it actually contains many (that is, $\Omega(m^p)$) copies of $K_p$. Contrapositively, this latter result says that if $\Gamma$ has few copies of $K_p$, then it cannot have substantially more edges than the Tur\'an graph. 

The result that we need, a combination of the two mentioned above, is the following. It asserts that if $\Gamma$ has slightly fewer edges than the Tur\'an graph (the stability regime) and has few copies of $K_p$ (the supersaturation regime), then it is close to the Tur\'an graph. 

\begin{theorem}\label{thm:stability}
	For every $\varepsilon>0$ and every integer $p \geq 2$, there exist $\eta,\alpha>0$ such that the following holds for all $m \geq 1$. Suppose $\Gamma$ is a graph on $m$ vertices with minimum degree at least $(1- \frac{1}{p-1}- \alpha)m$ and at most $\eta m^p$ copies of $K_p$. Then $V(\Gamma)$ can be partitioned into $V_1 \sqcup \dotsb \sqcup V_{p-1}$, such that the total number of internal edges in $V_1,\dotsc,V_{p-1}$ is at most $\varepsilon \binom m2$. 

	Moreover, we may take $\alpha= \min \{1/(2p^2),\varepsilon/2\}$ and $\eta = p^{-10p} \varepsilon$.
\end{theorem}
We were informed after the writing of this paper that \cref{thm:stability} can also be deduced from the work of Bollob\'as and Nikiforov \cite[Theorem 9]{BoNi} on joints in graphs, with slightly different quantitative dependencies.

A natural approach to prove \cref{thm:stability} is to first apply the celebrated graph removal lemma (see the survey \cite{MR3156927}). This allows us to pass to a $K_p$-free subgraph $\Gamma'$ of $\Gamma$ which still has very many edges. At this point, we can apply the standard stability theorem to deduce that $\Gamma'$ is nearly $(p-1)$-partite; since we deleted few edges to go from $\Gamma$ to $\Gamma'$, we must also have that $\Gamma$ is nearly $(p-1)$-partite. This proof technique was used to prove \cite[Corollary 3.4]{MR4170438}, which is a very similar result to \cref{thm:stability}. This proof technique actually proves a stronger theorem than \cref{thm:stability}, weakening the minimum degree assumption to an average degree assumption.

However, since the known bounds in the graph removal lemma are very weak, this proof technique would yield a tower-type dependence in the parameters $\varepsilon$ and $\eta$ in the statement of \cref{thm:stability}. Moreover, a super-polynomial dependence on the parameters is unavoidable if one only assumes an average degree condition. Indeed, let $\Gamma$ be the disjoint union of a Tur\'an graph on $(1- \alpha)m$ vertices and a graph $\Gamma_0$ on $\alpha m$ vertices which is extremal for the $K_p$ removal lemma, so that $\Gamma$ has at least $(1- \frac{1}{p-1}- \alpha)\binom m2$ edges. Then the distance of $\Gamma$ from being $(p-1)$-partite is roughly the same as the distance of $\Gamma_0$ from being $K_p$-free, and it is known that the clique removal lemma requires super-polynomial bounds in general \cite{Alon}. Such a construction shows that the clique removal lemma and stability-supersaturation theorems like \cref{thm:stability} are very closely related. 

The $\Gamma$ constructed has high average degree but low minimum degree, and this distinction turns out to be crucial. Indeed, in \cite{2105.09194}, the first and third authors proved that the $K_p$ removal lemma has \emph{linear} bounds if the minimum degree of $\Gamma$ is above a certain threshold, namely $(1-\frac{2}{2p-3})m$. This allows us to prove \cref{thm:stability} using the technique outlined above, while obtaining much stronger quantitative control. 

The first tool we need to prove \cref{thm:stability} is the high-degree removal lemma with linear bounds mentioned above, from \cite[Theorem 2.1]{2105.09194}. We remark that the explicit $p$-dependence of the constant is not given in \cite[Theorem 2.1]{2105.09194}, but it is easy to verify that the proof yields the following result. For completeness, we include this proof in \cref{sec:appendix}.
\begin{theorem}\label{thm:removal-above-threshold}
	Let $\Gamma$ be an $m$-vertex graph with with minimum degree at least $(1- \frac{2}{2p-3}+\beta)m$ and with at most $(10p)^{-2p} \beta \lambda m^p$ copies of $K_p$. Then $\Gamma$ can be made $K_p$-free by deleting at most $\lambda m^2$ edges.
\end{theorem}
\noindent We also use the following quantitative form of the stability theorem, due to F\"uredi \cite{MR3383250}.
\begin{theorem}\label{thm:furedi}
	Let $\Gamma$ be an $m$-vertex $K_p$-free graph with at least $(1- \frac{1}{p-1})\frac{m^2}{2}-\ell$ edges. Then $\Gamma$ can be made $(p-1)$-partite by deleting at most $\ell$ edges.
\end{theorem}
\noindent With these preliminaries, we can now prove \cref{thm:stability}.
\begin{proof}[Proof of \cref{thm:stability}]
Note that the result is vacuously true if $m<5$, as this implies that $\eta m^p<1$. So we henceforth assume that $m\geq 5$.
	Since $\alpha \leq 1/(2p^2)$, we see that
	\[
		1- \frac{1}{p-1} - \alpha \geq 1 - \frac{1}{p-1} -\frac{1}{2p^2} = 1- \frac{2}{2p-3} + \frac{5p-3}{4p^4-10p^3 + 6p^2} \geq 1 - \frac{2}{2p-3} + \frac{1}{p^3}.
	\]
	Therefore, we may apply \cref{thm:removal-above-threshold} with $\beta = 1/p^3$. We also set $\lambda = \varepsilon/10$, and note that the number of copies of $K_p$ in $\Gamma$ is at most 
	\[
		\eta  m^p = p^{-10p} \varepsilon m^p \leq (10p)^{-2p} \cdot \frac{1}{p^3} \cdot \frac \varepsilon {10} \cdot m^p = (10p)^{-2p} \beta \lambda m^p.
	\]
This implies that we may delete at most $\frac \varepsilon {10} m^2$ edges from $\Gamma$ to obtain a $K_p$-free graph $\Gamma'$. Since $\Gamma$ has minimum degree at least $(1- \frac{1}{p-1}- \alpha)n$, we see that $\Gamma'$ has at least $(1- \frac{1}{p-1} - \alpha)\frac{m^2}{2}- \frac \varepsilon {10}m^2$ edges. Therefore, by \cref{thm:furedi}, we see that $\Gamma'$ can be made $(p-1)$-partite by deleting at most $(\frac \alpha 2 + \frac\varepsilon {10})m^2$ edges. Let $\Gamma''$ be this $(p-1)$-partite subgraph, and let $V_1 \sqcup \dotsb \sqcup V_{p-1}$ be its $(p-1)$-partition. Since each $V_i$ is an independent set in $\Gamma''$, we see that the total number of edges of $\Gamma$ contained in $V_1,\dots,V_{p-1}$ is at most
	\[
		\frac \varepsilon {10}m^2 + \left(\frac \alpha 2 + \frac \varepsilon{10}\right)	m^2 \leq \left(\frac \varepsilon 5 + \frac \varepsilon 4\right)m^2 \leq \varepsilon \binom m2
	\]
	by our choice of $\alpha \leq \varepsilon/2$ and $m \geq 5$.
\end{proof}

\subsection{A blowup variant of the Andr\'asfai--Erd\H os--S\'os theorem}\label{sec:aes-variant}
The Andr\'asfai--Erd\H os--S\'os theorem \cite{AES} is a minimum-degree stability version of Tur\'an's theorem. It says that if an $m$-vertex $K_p$-free graph has minimum degree greater than $\frac{3p-7}{3p-4}m$, then it is $(p-1)$-partite; moreover, the constant $\frac{3p-7}{3p-4}$ is best possible. What we need is \cref{thm:aes-variant} instead, which we prove in this section. It says that if a graph has high minimum degree and does not contain some blowup of $K_p$, then it is $(p-1)$-partite. We remark that unlike Andr\'asfai, Erd\H os, and S\'os, we do not obtain the exact minimum degree threshold for being $(p-1)$-partite; for more on such refined questions, see e.g.\ \cite{Illingworth}.
 We need the following lemma, which is essentially due to Erd\H os \cite{Erdos64}.
\begin{lemma}\label{lem:zarankiewicz}
For every $0<\eta\leq \frac 12$ and $p,t \geq 2$, and $1\leq a_1 \leq \dotsb \leq a_{p-1} \leq t$, there exists some $\delta>0$ such that the following holds for large enough $m$. If $a_p \leq \delta m$ and $\Gamma$ is a $K_p(a_1,a_2,\dots,a_p)$-free graph on $m$ vertices, then $\Gamma$ has at most $\eta m^p$ copies of $K_p$.

Additionally, we may take $\delta \geq \eta^{t^{10p}}$, and the result holds for $m \geq 2t/\eta$.
\end{lemma}
\begin{proof}
We proceed by induction on $p$. The base case $p=2$ just says that a $K_{a_1,\delta m}$-free graph has at most $\eta m^2$ edges. So suppose that $\Gamma$ is an $m$-vertex $K_{a_1,\delta m}$-free graph with more than $\eta m^2$ edges. We double-count the number of copies of $K_{1,a_1}$ in $\Gamma$. On the one hand, every $a_1$-set has at most $\delta m$ common neighbors, so there are at most $\delta m \binom m{a_1} < \delta m^{a_1+1}$ copies of $K_{1,a_1}$. On the other hand, a vertex of degree $d$ contributes $\binom d{a_1}$ many copies. Therefore,
\[
\delta m^{a_1+1} > \sum_{v \in V(G)} \binom{\deg(v)}{a_1} \geq m \binom{2 e(\Gamma)/m}{a_1} \geq m \left(\frac{2e(\Gamma)}{a_1 m}\right)^{a_1}
\]
where the second inequality uses Jensen's inequality, which we may apply since $2e(\Gamma)/m \geq a_1$ by our assumption that $m\geq 2t/\eta$ is sufficiently large. Rearranging, we find that $e(\Gamma) < a_1\delta^{1/a_1} m^2/2$. If we let $\delta = (\eta/a_1)^{a_1}\geq \eta^{t^{20}}$, this gives the desired result.

We now proceed with the inductive step, so suppose that $\Gamma$ is an $m$-vertex $K_p(a_1,\dots,a_p)$-free graph with more than $\eta m^p$ copies of $K_p$. For every $(p-1)$-set of vertices $S$, let $\ext(S)$ denote the set of vertices $v$ such that $S \cup \{v\}$ is a $K_p$. Note that the sum of $\ab{\ext(S)}$ over all $(p-1)$-sets $S$ is exactly $p$ times the number of copies of $K_p$ in $\Gamma$. By assumption, this sum is therefore more than $p \eta m^p$. Thus, the average value of $\ab{\ext(S)}$ is greater than $p \eta m^p/\binom m{p-1}> \eta m$. Again by Jensen's inequality,
\[
\sum_{S \in \binom{V(\Gamma)}{p-1}} \binom{\ab{\ext(S)}}{a_1} > \binom{m}{p-1} \binom{\eta m}{a_1},
\]
and we may apply Jensen's inequality since $\eta m \geq a_1$ by our assumption that $m \geq 2t/\eta$ is sufficiently large.
Therefore, there is some $a_1$-set $A$ such that the common neighborhood of $A$ has at least 
\[
\binom{m}{p-1} \binom{\eta m}{a_1} / \binom m{a_1} \geq \eta' m^{p-1}
\]
copies of $K_{p-1}$, for some $\eta' \geq (\eta/2)^t/p^p$. We have that $(\eta')^{t^{10(p-1)}} \geq \eta^{t^{10p}}$, so by induction, the common neighborhood of $A$ must have a copy of $K_{p-1}(a_2,\dots,a_p)$, which is a contradiction.
\end{proof}

We can now prove \cref{thm:aes-variant}.
\begin{proof}[Proof of \cref{thm:aes-variant}]
Let $\eps=1/(100p^6 t^2)$, and let $\alpha,\eta$ be the parameters given in \cref{thm:stability}. Recall that $\alpha \leq \eps$. Finally, let $\delta=2^{-t^{100p}} < \eta^{t^{10p}}$. By \cref{lem:zarankiewicz}, we see that since $\Gamma$ is a $K_p(a_1,a_2,\dots,a_p)$-free graph on $m$ vertices, it must have at most $\eta m^p$ copies of $K_p$, and it has minimum degree at least $(1-1/(p-1)- \alpha)m$ by assumption. Therefore, \cref{thm:stability} implies that $\Gamma$ has a partition into parts $V_1,\dots,V_{p-1}$ such that the total number of internal edges is at most $\varepsilon \binom m2$. We fix such a partition with the minimum number of total internal edges. In particular, every vertex must have at least as many neighbors in every other part as it does in its own part.

Since $\Gamma$ has minimum degree at least $(1-1/(p-1)-\alpha)m$, it must have at least $(1-1/(p-1)-\alpha)\frac{m^2}{2}$ edges. Therefore, since there are at most $\eps \frac{m^2}{2}$ internal edges in $V_1,\dots,V_{p-1}$, we must have that
\begin{equation}\label{eq:edges-across-lb}
		\sum_{1 \leq i<j\leq p-1} e(V_i,V_j) \geq \left( 1- \frac{1}{p-1} -\alpha -\varepsilon \right) \frac{m^2} 2 \geq \left( 1- \frac{1}{p-1} - 2 \varepsilon \right) \frac{m^2}{2}
	\end{equation}
since $\alpha \leq \eps$.
We note that
\begin{equation}\label{eq:sum-of-squares}
	\sum_{i=1}^{p-1} \left( \ab{V_i}- \frac{m}{p-1} \right) ^2= \sum_{i=1}^{p-1} \ab{V_i}^2 - \frac{2m}{p-1} \sum_{i=1}^{p-1} \ab{V_i} + \frac{m^2}{p-1} = \sum_{i=1}^{p-1} \ab{V_i}^2 - \frac{m^2}{p-1}.
\end{equation}
Since the left-hand side of (\ref{eq:sum-of-squares}) is non-negative, we see that
\[
	\sum_{1 \leq i<j\leq p-1} \ab{V_i} \ab{V_j} = \frac12\left( m^2 - \sum_{i=1}^{p-1} \ab{V_i}^2 \right)\leq \frac 12\left( m^2 - \frac{m^2}{p-1} \right) = \left( 1- \frac{1}{p-1} \right) \frac{m^2}2.
\]
We can conclude from this that each $V_i$ has cardinality $\frac{m}{p-1} \pm  \sqrt{2\varepsilon}m$. For if not, then the left-hand side of (\ref{eq:sum-of-squares}) would be larger than $2 \varepsilon m^2$, and the above computation would contradict (\ref{eq:edges-across-lb}).

Now, suppose that for some $1\leq a < b\leq p-1$, we have that $e(V_a,V_b) <(1- p^2\varepsilon)\ab{V_a}\ab{V_b}$. Then we would find that
\[
	\sum_{1 \leq i<j\leq p-1} e(V_i,V_j) < \sum_{1 \leq i<j \leq p-1} \ab{V_i}\ab{V_j} -p^2 \varepsilon \ab{V_a} \ab{V_b} \leq \left( 1- \frac{1}{p-1}-2 \varepsilon \right) \frac{m^2}{2},
\]
contradicting (\ref{eq:edges-across-lb}), using the bound $\ab{V_a} \geq \frac{m}{p-1}-\sqrt{2 \varepsilon}m \geq m/p$ since $\eps \leq 1/(2p^4)$. Therefore, we find that for all $i \neq j$,
\begin{equation}\label{eq:pairs-dense}
	e(V_i,V_j) \geq (1- p^2 \varepsilon) \ab{V_i} \ab{V_j}.
\end{equation}
Now suppose that some vertex $v \in V_i$ has more than $2p^2\sqrt \eps \ab{V_i}$ neighbors in its own part $V_i$. By our assumption above, this means that $v$ also has more than $2p^2 \sqrt \eps \ab{V_i}\geq p^2 \sqrt \eps \ab{V_j}$ neighbors in each part $V_j$ for $j \neq i$, where we used the fact that $\ab{V_j} = \frac{m}{p-1} \pm \sqrt{2\eps}m$ and the fact that $\eps\leq 1/(2p^4)$ to conclude that $\ab{V_i} \geq \frac 12 \ab{V_j}$. Let $U_j = N(v) \cap V_j$ denote the neighbors of $v$ in $V_j$. For every $1 \leq a \neq b \leq p-1$, we have by (\ref{eq:pairs-dense}) that
\[
    e(U_a,U_b) \geq \ab{U_a} \ab{U_b} - p^2 \eps \ab{V_a} \ab {V_b} \geq \left(1 - \frac{p^2 \eps}{p^4 \eps}\right) \ab{U_a} \ab{U_b} =\left(1 - \frac 1{p^2}\right) \ab{U_a} \ab{U_b},
\]
where the second inequality uses our assumption that $\ab{U_a} \geq p^2 \sqrt \eps \ab{V_a}$, and similarly for $U_b$. By the union bound, if we pick a random vertex from $U_a$ for each $1 \leq a \leq p-1$, then they span a copy of $K_{p-1}$ with probability at least $1-\binom p2/p^2\geq \frac12$. Therefore, the neighborhood of $v$ contains at least
\[
\frac 12 \prod_{a=1}^{p-1} \ab{U_a} \geq \frac{(p^2 \sqrt \eps)^{p-1}}{2}\prod_{a=1}^{p-1} \ab{V_a} \geq \frac{(p^2\sqrt \eps)^{p-1}}{2p^{p-1}} m^{p-1} = \eta' m^{p-1}
\]
copies of $K_{p-1}$, where $\eta' \geq (20p^2 t)^{-p}$ by our choice of $\eps$.
By \cref{lem:zarankiewicz}, and by our choice of $\delta$, the neighborhood of $v$ contains a copy of $K_{p-1}(a_2,\dots,a_p)$. Since $a_1=1$, this implies that $\Gamma$ contains a copy of $K_p(a_1,a_2,\dots,a_p)$, which is a contradiction. Thus, we conclude that every vertex $v \in V_i$ has at most $2p^2 \sqrt \eps \ab{V_i}$ neighbors in its own part $V_i$, for every $1 \leq i \leq p-1$. 

We now claim that for every $1 \leq i \neq j \leq p-1$, every vertex $v \in V_i$ has at least $(1-1/(2pt))\ab{V_j}$ neighbors in $V_j$. Indeed, if not, then $v$ has at least $\ab{V_j}/(2pt)$ non-neighbors in $V_j$, and at least  $(1-2p^2 \sqrt \eps)\ab{V_i}-1$ non-neighbors in $V_i$. In total, the number of non-neighbors of $v$ is at least
\[
\frac{1}{2pt} \ab{V_j} + (1-2p^2 \sqrt \eps)\ab{V_i}-1 > \left(1 + \frac 1{2pt} - 2p^2 \sqrt \eps - 4\sqrt \eps\right)\frac m{p-1} > \left(\frac 1{p-1} + \eps\right)m,
\]
where the final inequality uses our choice of $\eps$. This contradicts the assumption that the minimum degree of $\Gamma$ is at least $(1-1/(p-1)-\eps)m$.

Now suppose that there is some edge $vw$ inside some part $V_i$, and assume without loss of generality that $i=1$. The vertices $v$ and $w$ have at least $(1-1/(pt))\ab{V_2}>a_3$ common neighbors in $V_2$, so we may pick some set of $a_3$ common neighbors in $V_2$. Then $v,w,$ and these $a_3$ common neighbors have at least $(1-(a_3+2)/(2pt))\ab{V_3}>(1-2t/(2pt))\ab{V_3}>a_4$ common neighbors in $V_3$, so we may pick $a_4$ such common neighbors in $V_3$. Continuing in this way, we can greedily pick $a_{j+1}$ vertices from $V_j$ which are common neighbors of the previously chosen vertices, for each $j \leq p-2$. Having done this, we have picked at most $pt$ vertices, so they still have at least $(1-pt/(2pt))\ab{V_{p-1}} = \frac 12 \ab{V_{p-1}}> \delta m$ common neighbors in $V_{p-1}$. Thus, we have built a copy of $K_p(a_1,\dots,a_p)$ in $\Gamma$, a contradiction. This shows that there can be no edge inside any $V_i$, and thus that $\Gamma$ is $(p-1)$-partite.
\end{proof}

\subsection{Proof of Theorem \ref{thm:multipartite}\ref{item:if}}\label{sec:proof-multipartite}
We begin by establishing an upper bound on the Ramsey number of a complete multipartite graph vs.\ a fixed clique. This result can be viewed as a weak version of Ramsey goodness, as it implies that a certain off-diagonal Ramsey number grows linearly in the number of vertices of the first graph. We need such a bound in order to apply \cref{lem:find-blowup}, which we will do shortly in the proof of \cref{lem:low-deg-vtxs}.
\begin{lemma}\label{lem:off-diag-bound}
    For all $\gamma>0$ and all integers $k,p,t,z\geq 2$, there exists $\delta>0$ such that for all $1 \leq a_1 \leq \dotsb \leq a_{p-1} \leq t$, all $n \geq 2t (e\cdot r(K_p,K_z))^p/\gamma$, and all $a_p \leq \delta n$, we have
    \[
    r(K_p(a_1,\dots,a_p), K_z) \leq \gamma n.
    \]
    Moreover, we may take $\delta \geq \gamma (e\cdot r(K_p,K_z))^{-t^{20p}}$.
\end{lemma}
\begin{proof}
    Let $m=\gamma n$ and $r=r(K_p,K_z)$. Let $\eta = (er)^{-p}$ and $\delta=\gamma \eta^{t^{10p}}\geq \gamma (er)^{-t^{20p}}$. Note that $m \geq 2t/\eta$ and that if $a_p \leq \delta n$, then $a_p \leq \eta^{t^{10p}} m$. Let $\Gamma$ be an $m$-vertex graph with no independent set of order $z$; we wish to prove that $\Gamma$ contains a copy of $K_p(a_1,\dots,a_p)$. We use an averaging argument originally due to Erd\H os \cite{MR151956}.
    
    Since $\Gamma$ has no independent set of order $z$, every set of $r$ vertices must contain a copy of $K_p$, by the definition of the Ramsey number $r=r(K_p,K_z)$. Moreover, every copy of $K_p$ appears in precisely $\binom{m-p}{r-p}$ many $r$-sets, so the number of copies of $K_p$ in $\Gamma$ is at least
    \[
    \frac{\binom mr}{\binom{m-p}{r-p}} = \frac{\binom mp}{\binom rp} \geq \frac{1}{(er)^p} \cdot m^p = \eta m^p.
    \]
    By (the contrapositive of) \cref{lem:zarankiewicz}, this implies that $\Gamma$ contains a copy of $K_p(a_1,\dots,a_p)$, where we may apply \cref{lem:zarankiewicz} because $a_p \leq \eta^{t^{10p}} m$ and $m \geq 2t/\eta$. This concludes the proof.
\end{proof}

We are now ready to begin the proof of \cref{thm:multipartite}\ref{item:if}. Recall that $G=K_p(a_1,\dots,a_p)$ and $H=\book kn$, and we are interested in studying a graph $\Gamma$ on $N=(p-1)(n-1)+1$ vertices such that $\Gamma$ is $G$-free and $\ol \Gamma$ is $H$-free.
We begin by proving that most vertices of $\Gamma$ have high degree.

\begin{lemma}\label{lem:low-deg-vtxs}
For all integers $k,p,t\geq 2$ and all $\alpha\in (0,\frac 12]$, there exists some $\delta>0$ such that the following holds for all $1 =a_1=a_2 \leq a_3 \leq \dotsb \leq a_{p-1} \leq t$, all sufficiently large $n$, and all $a_p \leq \delta n$. Let $G=K_p(a_1,\dots,a_p)$ and $H=\book kn$.
If $\Gamma$ is a graph on $N = (p-1)(n-1)+1$ vertices such that $\Gamma$ is $G$-free and $\overline{\Gamma}$ is $H$-free, then at most $\alpha N$ vertices of $\Gamma$ have degree at most $d = (1-1/(p-1) - \alpha) N$.

Additionally, we may take $\delta \geq 2^{-(t/\alpha)^{100kp}}$, and the result holds for all $n \geq 2^{(t/\alpha)^{100kp}}$.
\end{lemma}
\begin{proof}

Let $\eps = (\alpha/10)^k/(10k)$ be small in terms of $k$ and $\alpha$, and let $x = t/\eps$ and $z = x(20k)^p$ be large in terms of $k,t$, and $\alpha$. 
We now let $\gamma = \alpha z^{-z}$ and $$\delta = \min\{\gamma (e\cdot r(K_p,K_z))^{-t^{20p}}, (\alpha/10)^k\cdot (2\eps/e)^{tp}\}$$ 
be small in terms of $k,t,p$, and $\alpha$. Finally, we let $n_0 = 2t (e\cdot r(K_p,K_z))^p/\gamma$. It is straightforward to verify that $\delta \geq 2^{-(t/\alpha)^{100kp}}$ and $n_0\leq 2^{(t/\alpha)^{100kp}}$. We henceforth assume that $n \geq n_0$, and let $\Gamma$ be a graph on $N=(p-1)(n-1)+1$ vertices satisfying the assumptions of the lemma. Note that by our choices of $n_0$ and $\delta$, we are in a position to apply \cref{lem:off-diag-bound}, which implies that $r(G,K_z)<\gamma n = \alpha n/z^z$.

Let $S \subset V(\Gamma)$ be the set of vertices of degree
less than $d$. We proceed by contradiction and assume that $|S| \ge \alpha N$. This implies that $\ab S \geq z^z r(G,K_z)$, so we may apply \cref{lem:find-blowup}. We find that there is an induced copy of $K_{p-1}(x)$ in $\Gamma$ whose parts are $V_1,\ldots, V_{p-1}$ with $V_1 \subseteq S$. Thus, all the vertices of $V_1$ have degree less than $d$. 

Partition the vertices of $\Gamma$ into $p$ parts $U_0,\ldots, U_{p-1}$, where for $i\ge 1$,  each vertex in $U_i$ has at most $\eps x$ neighbors in $V_i$, and $U_0$ consists of all vertices with more than $\eps x$ neighbors in each $V_i$. 

First, if $|U_0|\ge (e/\eps)^{tp} \cdot \delta n$, then we can find a copy of $G$ in $\Gamma$ as follows. If we pick a set $W$ by taking $a_i$ vertices uniformly at random from $V_i$ for $i=1,\ldots,p-1$, then the expected number of vertices of $U_0$ complete to $W$ is at least
\[
\prod_{i=1}^{p-1} \frac{\binom{\eps x}{a_i}}{\binom{x}{a_i}} \cdot |U_0| \ge (\eps/e)^{tp}\cdot |U_0| \ge \delta n \geq a_p,
\]
where the first inequality holds since $\eps x = t \geq a_i$ for all $1 \leq i \leq p-1$.
Thus, there exists a $W \subseteq V_i$ for which we can find $a_p$ vertices of $U_0$ which together with $W$ form a copy of $G= K_p (a_1, \ldots, a_p)$. This is a contradiction.

Next, suppose $|U_i| \ge (1-2k \eps)^{-1} (n-k)$ for some $i\ge 1$. Every vertex in $U_i$ has at most $\eps x$ neighbors in $V_i$, so we may remove half the vertices of $V_i$ (the ones with highest degree to $U_i$) to find a subset $V'_i$ such that every $v\in V'_i$ has at most $2\eps |U_i|$ neighbors in $U_i$. Note that $\ab{V_i'} \geq x/2 \geq k$ by our choice of $x$. Take $W$ to be any $k$-subset of $V'_i$, and let $U'_i \subseteq U_i$ be the set of vertices with no neighbors in $W$. We have $|U'_i| \ge (1-2 k\eps) |U_i|\ge n-k$, and so $W$ and $U'_i$ form a copy of $\book kn$ in $\ol \Gamma$, which is again impossible. We henceforth assume that $|U_i| < (1-2k \eps)^{-1}(n-k)$ for all $i\geq 1$.

Finally, suppose $|U_1| \ge (1-2k\eps)^{-1}(n-k) - (\alpha /10)^k N.$
We seek to find a copy of $\book kn$ in $\ol \Gamma$ again, this time using the degree condition on $V_1$. As before, we may pass to a subset $V'_1$ of half the vertices of $V_1$ such that each has at least $(1-2\eps) |U_1|$ non-neighbors in $U_1$. Each vertex of $V'_1$ has degree at most $d = (1-1/(p-1)-\alpha)N$, and so has at least $N-1-d = N/(p-1)+\alpha N -1 \geq n + \alpha N-2$ non-neighbors in total. In particular, since $|U_1| < (1-2k \eps)^{-1}(n-k) < (1+3k\eps) n-2$ and $\alpha \ge 6k \eps$, each vertex of $V'_1$ has at least $(n + \alpha N -2) - |U_1| \ge \alpha N/2$ non-neighbors in $\ol{U_1}$.

Pick a random $k$-subset $W$ of $V'_1$ to form the spine of the book. The number of common non-neighbors the vertices of $W$ have inside $U_1$ is at least $(1-2k\eps)|U_1|$. We now count the expected number of common non-neighbors the vertices of $W$ have in $\ol{U_1}$. For the convenience of the following calculation, we insert phantom vertices to $\ol{U_1}$, each complete to $W$, until $|\ol{U_1}| = N$; this has no effect on the common non-neighborhood we care about. If $u \in \ol{U_1}$ has $y$ non-neighbors in $V'_1$, then the probability that $W$ is chosen entirely among these $y$ vertices is $\binom{y}{k}/\binom{x/2}{k}$.
Since vertices in $V'_1$ have at most $(1-\alpha/2)N$ neighbors in $\ol{U_1}$, the average value of $y$ over a random $u\in \ol{U_1}$ is at least $\alpha/2 \cdot |V'_1| = \alpha x/4$. By linearity of expectation and convexity we find that the expected number of common non-neighbors of $W$ in $\ol{U_1}$ is at least 
\[
\binom{\alpha x / 4}{k}\binom{x/2}{k}^{-1} |\ol{U_1}| \ge \left(\frac{\alpha x}{4k}\right)^k\left(\frac{2k}{ex}\right)^k |\ol{U_1}| \ge \left(\frac\alpha  {2e}\right)^k \cdot \frac N 2 \ge \left(\frac \alpha  {10}\right)^k N.
\] Thus, there exists some particular $W$ with at least $(1-2\eps)|U_1| + (\alpha/10)^k N \ge n-k$ non-neighbors, forming the desired $\book kn$ in $\ol \Gamma$. This contradicts our assumptions on $\Gamma$.

We conclude that the partition $V(\Gamma) = U_0\sqcup \cdots \sqcup U_{p-1}$ satisfies
\begin{align*}
    |U_0|& < (e/\eps)^{tp} \cdot \delta n, \\
    |U_1| & < \frac{n-k}{1-2k\eps} - (\alpha /10)^k N, \\
    |U_i| & < \frac{n-k}{1-2k \eps} \qquad\qquad\textnormal{ if $i \ge 2$.}
\end{align*}
Adding these together, we obtain that the number $N$ of vertices in $\Gamma$ is
\begin{align*}
    N=\sum_{i=0}^{p-1} |U_i| & < \left(\frac e \eps\right)^{tp} \cdot \delta n + (p-1) \frac{n-k}{1-2k \eps} - \left(\frac \alpha {10}\right)^k N \\
    & < (1+ 3k\eps) N + \left(\frac e \eps\right)^{tp} \cdot \delta n - \left(\frac \alpha {10}\right)^k N \\
    &< \left( 1 + \frac 12 \left(\frac \alpha{10}\right)^k\right) N + \left(\frac e \eps\right)^{tp} \cdot \left(\frac \alpha{10}\right)^k \cdot \left(\frac{2 \eps}e\right)^{tp} N - \left(\frac \alpha{10}\right)^k N\\
    &< N.
\end{align*}
This is a contradiction and we are done.
\end{proof}

We now have all the tools to complete the proof of \cref{thm:multipartite}\ref{item:if}.

\begin{proof}[Proof of \cref{thm:multipartite}\ref{item:if}.]
Let $\alpha=1/(200kt^2p^6)$ and $\delta = 2^{-(t/\alpha)^{100kp}}\geq 2^{-t^{1000k^2p^2}}$ be the parameter from \cref{lem:low-deg-vtxs}. The result we wish to prove is vacuously true if $n<1/\delta$, so we assume henceforth that $n \geq 1/\delta= 2^{(t/\alpha)^{100kp}}$.
Recall that $H=\book kn$, and $G = K_p(a_1,\dots,a_p)$, where $1 \leq a_1 \leq \dotsb \leq a_{p-1} \leq t$ and $a_p \leq \delta n$. We are given a $G$-free graph $\Gamma$ on $N = (p-1)(n-1)+1$ vertices, and we wish to show that $\ol \Gamma$ contains a copy of $H$. By \cref{lem:low-deg-vtxs}, we have that at most $\alpha N$ vertices of $\Gamma$ have degree at most $d = (1-1/(p-1)- \alpha)N$. If we let $T$ be the set of vertices of degree greater than $d$, then the induced subgraph $\Gamma[T]$ has at least $(1-\alpha)N$ vertices and thus minimum degree at least $(1-1/(p-1)-2 \alpha)\ab T$. We now apply \cref{thm:aes-variant} to the graph $\Gamma[T]$, which we may do since $2\alpha \leq 1/(200p^6 t^2)$, $\delta \leq 2^{-t^{100p}}$, and $(1-\alpha)N \geq N/2 \geq tp^{20p}$.
Doing so, we find that $\Gamma[T]$ is $(p-1)$-partite. Let the parts of $\Gamma[T]$ be $T_1,\dots,T_{p-1}$. We now argue roughly as in the proof of \cref{thm:aes-variant}.

Recall that for a vertex $v$ and a vertex set $W$, we denote by $d(v,W)$ the \textit{density} of $v$ to $W$, namely the number of neighbors of $v$ in $W$ divided by $\ab{W}$.
\begin{claim}\label{cl:partition-properties}
Let $T_1,\dots,T_{p-1}$ be as defined above. Let $\xi = 4p^2 \alpha$. Then for every $1 \leq i \neq j \leq p-1$, we have that
\begin{equation}\label{eq:part-sizes}
    \left(\frac{1}{p-1} -\xi\right)N \leq \ab{T_i} \leq \left(\frac{1}{p-1} +\xi\right)N
\end{equation}
and
\begin{equation}\label{eq:pair-densities}
    d(w,T_j) \geq 1-\xi \text{ for every }w \in T_i.
\end{equation}
\end{claim}
\begin{proof}
    Since $T_i$ is an independent set, every vertex in $T_i$ has degree at most $N-\ab{T_i}$. Since every vertex in $T_i$ has degree at least $d$, this implies that $\ab{T_i} \leq N-d = (1/(p-1)+\alpha)N$. Since $T_1,\dots,T_{p-1}$ partition $T$, which has size at least $(1-\alpha)N$, this  implies that $\ab{T_i} = \ab{T} -\sum_{j \neq i} \ab{T_j} \geq (1/(p-1)-p \alpha)N$, which proves (\ref{eq:part-sizes}) since $p \alpha < \xi$.
    
    For (\ref{eq:pair-densities}), we recall that the induced subgraph $\Gamma[T]$ has minimum degree at least $(1-1/(p-1)-2 \alpha)\ab{T}$. So any $w \in T_i$ has at most $(1/(p-1)+2 \alpha)\ab T$ non-neighbors in $T$. Additionally, since $T_i$ is an independent set, every $w \in T_i$ has $\ab{T_i}-1$ non-neighbors in $T_i$. If $d(w,T_j)<1-\xi$, then the total number of non-neighbors of $w$ is at least
    \[
    \xi \ab{T_j} + \ab{T_i} - 1 \geq (1+\xi)\left( \frac{1}{p-1} - p \alpha\right)N - 1 > \left( \frac{1+\xi}{p-1} - 2p \alpha\right) N > \left(\frac 1{p-1} + 2\alpha\right)\ab{T},
    \]
    using the computations above and our choice of $\xi = 4p^2 \alpha$. This is a contradiction.
\end{proof}
Let $S$ be the complement of $T$, i.e.\ the set of vertices in $\Gamma$ with degree less than $d$, and recall that $\ab{S} \leq \alpha N$. 

\begin{claim}
Let $\zeta =pt \xi =  4p^3 t \alpha$. For every $v \in S$, at least one of the following is true. Either $v$ has no edges to some $T_i$, or else $d(v,T_i) < \zeta$ for at least two different choices of $i \in [p-1]$.
\end{claim}
\begin{proof}
    Suppose for contradiction that this is false for some $v \in S$. Thus, $d(v,T_i)\geq \zeta$ for all but at most one choice of $i \in [p-1]$, and additionally $v$ has a neighbor in each $T_i$. By relabeling the parts, we may assume that $d(v,T_i) \geq \zeta$ for all $i \in [p-2]$. Let $w$ be a neighbor of $v$ in $T_{p-1}$. By (\ref{eq:pair-densities}), we see that $v$ and $w$ have at least $(\zeta-\xi)\ab{T_1} > \xi\ab{T_1} > a_3$ common neighbors in $T_1$. Pick any $a_3$ common neighbors in $T_1$. Then $v,w$, and these $a_3$ common neighbors have at least $(\zeta - (a_3+1)\xi)\ab{T_2} > \xi\ab{T_2} > a_4$ common neighbors in $T_2$. Continuing in this way, we can pick out $a_i$ vertices in $T_{i+2}$ which are common neighbors of all previously-chosen vertices. At the end of this process, we can still pick at least $(\zeta-(p-1)t\delta)\ab{T_{p-2}} \geq \delta n$ common neighbors in $T_{p-2}$, and thus we can build a copy of $G$, contradicting our assumption that $\Gamma$ is $G$-free.
\end{proof}
We partition $S$ into $S_1 \cup S_2$, where $S_1$ consists of all vertices in $S$ that are empty to some part $T_i$, and $S_2$ consists of the remaining vertices $v$, namely those satisfying $d(v,T_i)<\zeta$ for at least two choices of $1 \leq i \leq p-1$.

Now, we pick an index $i \in [p-1]$ uniformly at random, and then pick a $k$-set $Q \subset V_i$ uniformly at random. By doing so, we obtain a (non-uniform) distribution on the set of $k$-cliques in $\ol \Gamma$. For a vertex $v \in V(\Gamma)$, let us say that $v$ \emph{extends} $Q$ if $Q \cup \{v\}$ is also a clique in $\ol \Gamma$, or equivalently if $v$ is not adjacent in $\Gamma$ to any vertex of $Q$. Note that if $v \in Q$, then we still say that $v$ extends $Q$, even though this is not really an extension per se. We observe that if $v \in T$, then the probability that $v$ extends $Q$ is at least $1/(p-1)$. Indeed, the probability that $v$ extends $Q$ is at least the probability that $v \in T_i$ for the randomly chosen index $i$, which is exactly $1/(p-1)$ since we pick the index $i$ uniformly at random.

Next, if $v \in S_1$, then we again have that the probability that $v$ extends $Q$ is at least $1/(p-1)$. Indeed, if $v \in S_1$, then $v$ has no edges to $T_j$ for at least one index $j$. The probability that $v$ extends $Q$ is then at least the probability that $j$ is the randomly chosen index, which equals $1/(p-1)$.

Finally, if $v \in S_2$, then without loss of generality, $d(v,T_1) < \zeta$ and $d(v,T_2)<\zeta$. If the randomly chosen index $i$ is $1$ or $2$, then the probability that $v$ has an edge to $Q$ is at most $k\zeta$, by the union bound. Therefore, if $v \in S_2$, then
\[
\pr(v\text{ extends }Q) \geq \frac{2}{p-1}\cdot (1-k\zeta) \geq \frac{1}{p-1},
\]
since we chose $\alpha$ so that $k\zeta = 4kp^3t\alpha \leq 1/2$. By putting all of this together, we find that $\pr(v \text{ extends }Q) \geq 1/(p-1)$ for every vertex $v \in V(\Gamma)$. By linearity of expectation, this implies that
\[
\E[|\{v : v\text{ extends }Q\}|] = \sum_{v \in V(\Gamma)} \pr(v\text{ extends }Q) \geq \frac{N}{p-1}.
\]
Therefore, there exists some clique $Q$ in $\ol \Gamma$ which has at least $N/(p-1)> n-1$ extensions. Since exactly $k$ of these extensions are the degenerate ones coming from vertices in $Q$ itself, we find that $\ol \Gamma$ contains a copy of $H=\book kn$. This completes the proof.
\end{proof}

\section{Concluding remarks}\label{sec:concluding}

In this section we collect a few of the tantalizing open questions remaining in this area. 
\paragraph{Removing regularity.}
Note that the full Ramsey goodness results of Nikiforov and Rousseau~\cite{NR09} hold in greater generality than our results \cref{thm:main1} and \cref{thm:multipartite}. However, due to the dependence of their arguments on Szemer\'edi's regularity lemma, the quantitative dependence between the graph sizes involved are tower-type. It would be interesting to find a direct proof of their goodness results without regularity, as this would likely lead to superior quantitative bounds.

\paragraph{Near Ramsey goodness.}
In \cref{thm:multipartite}, we study the Ramsey number $r(K_p(a_1,\dots,a_p), \book kn)$ for sufficiently large $n$, where $a_1,\dots,a_{p-1}$ are fixed and $a_p \leq \delta n$ for some absolute constant $\delta>0$. We are able to determine this Ramsey number in the case $a_1 = a_2 = 1$ (in which case the answer is given by the Ramsey goodness bound), but it is natural to ask what happens for larger values of $a_1$ and $a_2$. In this case, there is a natural lower bound, generalizing the proof of the ``only if'' direction of \cref{thm:multipartite}, and which shows a surprising connection to an analogue of the classical extremal problem for complete bipartite graphs. To explain this connection, we first define the following Dirac-type extremal function.
\begin{definition}
Given a graph $H$ and integers $k,n$, let $d_k(n,H)$ be the maximum $d$ for which there is an $(n+d-1)$-vertex $H$-free graph, at most $k-1$ vertices of which have degree less than $d$.
\end{definition}
Now let $d = d_k(n, K_{a_1,a_2})$, and let $\Gamma_0$ be a $K_{a_1,a_2}$-free graph on $n+d-1$ vertices, at most $k-1$ of which have degree less than $d$. Let $\Gamma$ be a graph with $N=(p-1)(n-1)+d$ vertices, whose vertex set is divided into $p-1$ parts $U_1,\dots,U_{p-1}$ with $\ab{U_1} = n+d-1$ and $\ab{U_2} = \dotsb = \ab{U_{p-1}} = n-1$, such that $\Gamma[U_1]$ is isomorphic to $\Gamma_0$, and such that all pairs of vertices in different parts are adjacent. Then $\Gamma$ is $K_p(a_1,\dots,a_p)$-free, since $U_2,\dots,U_{p-1}$ are independent sets, and $\Gamma[U_1]$ is $K_{a_1,a_2}$-free. Additionally, $\ol \Gamma$ is a disjoint union of $\ol{\Gamma_0}$ and $p-2$ cliques of order $n-1$. The cliques are too small to contain a copy of $\book kn$, and all but at most $k-1$ vertices of $\ol{\Gamma_0}$ have degree at most $(n+d-1)-1-d = n-2$. Since $\book kn$ has $k$ vertices of degree $n-1$, this shows that $\ol \Gamma$ is $\book kn$-free. Thus, we conclude that
\begin{equation}\label{generalextremal}
r(K_p(a_1,\dots,a_p), \book kn) > (p-1)(n-1) + d_k(n,K_{a_1,a_2}).
\end{equation}
Our proof of \cref{thm:multipartite}\ref{item:only-if} used the same argument, and we simply noted that if $a_2>1$, then for sufficiently large $n$, we have $d_k(n,K_{a_1,a_2})\geq a_1$ for all $k \geq 2$. We conjecture that the lower bound (\ref{generalextremal}) is tight for sufficiently large $n$, if $a_1,\dots,a_{p-1}$ are fixed, and $a_p \leq \delta n$.
\begin{conjecture}\label{conj:dirac-type}
    For all integers $k,p,t\geq 2$, there exists some $\delta>0$ such that the following holds for all $n \geq 1$. For positive integers $a_1 \leq \dotsb \leq a_{p-1} \leq t$ and $a_p \leq \delta n$, we have
    \[
    r(K_p(a_1,\dots,a_p), \book kn) = (p-1)(n-1) + d_k(n,K_{a_1,a_2})+1.
    \]
\end{conjecture}
\noindent Thus, \cref{thm:multipartite} verifies \cref{conj:dirac-type} in the case $a_1 = a_2 = 1$.
\paragraph{Disconnected graphs.}
Ramsey goodness results are some of the rare examples in graph Ramsey theory where exact values of Ramsey numbers are known. Another such example is an old result of Burr, Erd\H os, and Spencer~\cite{BuErSp}, recently improved by Buci\'c and Sudakov~\cite{BuSu}, which shows
\[
r(nG, nG) = 2(|V(G)| - \alpha(G))n + c
\]
for $n$ sufficiently large and some constant $c=c(G)$. Here, $G$ is a fixed graph, $nG$ is a vertex disjoint union of $n$ copies of $G$, and $\alpha(G)$ is the independence number of $G$. Does there exist a theory of Ramsey goodness for disconnected graphs, giving a common generalization of the Burr--Erd\H os--Spencer result and our theorems?

\paragraph{Empty pairs in triangle-free graphs.}
Motivated by a well-studied approach to the famous Erd\H os--Hajnal conjecture, the following conjecture was proposed by Conlon, Fox, and Sudakov.
    \begin{conjecture}[{\cite[Conjecture 3.14]{CoFoSu15}}] \label{conj:triangle-free}
        There exists some $\eps>0$ such that every $N$-vertex triangle-free graph contains two vertex subsets $A,B$ with $\ab A \geq \eps N$, $\ab B \geq N^\eps$, and with no edges between $A$ and $B$.
    \end{conjecture}
    For more on this conjecture and its variants, see also \cite{CFSSS}. \cref{conj:triangle-free} remains open. The strongest result in this direction, due independently to Fox and Shapira (unpublished) says that one may take $\ab A \geq \eps N$ and $\ab B \geq \eps\log N/\log \log N$. One consequence of \cref{thm:main1} is that we may take $\ab A \geq \eps N$ and $\ab B \geq (\log N)^\eps$, for $\eps = 1/31$. Indeed, \cref{thm:main1} with $p=3$ says that if $n \geq 2^{k^{10p}} = 2^{k^{30}}$ and if $N = (p-1)(n-1)+1 = 2n-1$, then for every $N$-vertex triangle-free graph $\Gamma$, its complement $\ol \Gamma$ contains a copy of $\book kn$. Let $A$ be the set of leaves of this book and $B$ be its spine, so that $\ab A = n \geq N/31$ and $\ab B = k \geq (\log N)^{1/31}$. Since $A \cup B$ span a book in $\ol \Gamma$, there are no edges between $A$ and $B$ in $\Gamma$.
    
    By the same argument, we see that improving the bounds in \cref{thm:main1} could yield progress on \cref{conj:triangle-free}. For example, improving the bound $n \geq 2^{k^{10p}}$ in \cref{thm:main1} to a bound that is single-exponential in both $k$ and $p$ would allow one to take $\ab A \geq \eps N$ and $\ab B \geq \eps \log N$ in \cref{conj:triangle-free}.

\paragraph{Ramsey goodness threshold.}    
 More generally, it is natural to ask what the ``Ramsey goodness threshold'' is in \cref{thm:main1}. That is, what is the smallest $n$ (in terms of $k$ and $p$) such that $r(K_p, \book kn) = (p-1)(n-1)+1$? A simple random construction shows that this threshold is at least $(k/\log p)^{cp}$, for an absolute constant $c>0$. Indeed, let $n = (k/\log p)^{cp}$ and $N = (p-1)(n-1)+1$, and let $\Gamma$ be an Erd\H os--R\'enyi random graph on $N$ vertices with edge probability\footnote{If the quantity $C(\log p)/k$ is greater than $1$, then the result we are trying to prove is vacuously true, since $(k/\log p)^{cp}$ is then less than $1$. Thus we may assume that this is a valid edge probability.} $C(\log p)/k$, for an absolute constant $C>0$. Then a first moment estimate shows that with positive probability, $\Gamma$ does not contain a copy of $K_p$ and its complement does not contain a copy of $\book kn$.
    
    However, there remains a rather large gap between the lower bound of $(k/\log p)^{cp}$ and the upper bound of $2^{k^{10p}}$ for this threshold. In particular, it would be interesting to determine if, for $p$ fixed, the correct behavior is polynomial or exponential in $k$.

\appendix 
\section*{Appendix}
\section{A complete proof of Theorem \ref{thm:removal-above-threshold}}\label{sec:appendix}

\cref{thm:removal-above-threshold} was proved in \cite[Theorem 2.1]{2105.09194}, but the explicit $p$-dependence on the constant was not computed there. So for completeness, we recreate here the proof of \cite[Theorem 2.1]{2105.09194}, while explicitly keeping track of the constant. As pertains to all of the main ideas, the proof is identical to that of \cite{2105.09194}.

We recall the statement of \cref{thm:removal-above-threshold}.
\begin{thm:removal-above-threshold}
	Let $\Gamma$ be an $m$-vertex graph with with minimum degree at least $(1- \frac{2}{2p-3}+\beta)m$ and with at most $(10p)^{-2p} \beta \lambda m^p$ copies of $K_p$. Then $\Gamma$ can be made $K_p$-free by deleting at most $\lambda m^2$ edges.
\end{thm:removal-above-threshold}

As in \cite{2105.09194}, it suffices to prove the following key lemma, which is a restatement of \cite[Lemma 2.3]{2105.09194} with an explicit constant.
\begin{lemma}\label{lem:removal-lemma}
	Let $p \geq 3$ and $\beta>0$. If $\Gamma$ is an $m$-vertex graph with minimum degree at least $(1-\frac{2}{2p-3} + \beta)m$, then every $K_p$ in $\Gamma$ contains an edge which lies in at least $\beta (m/(8p))^{p-2}$ copies of $K_p$.
\end{lemma}

In the proof of \cref{lem:removal-lemma}, we will use the following simple inequality, proved in \cite[Lemma 2.2]{2105.09194}.

\begin{lemma}\label{lem:calc-inequality}
	For any $x \geq 4$, we have that
	\[
		x \frac{2x-5}{2x-3} \geq x-2 + \frac 25.
	\]
\end{lemma}
We will also use the following inequality.
\begin{lemma}\label{lem:ineq-two}
    For all positive integers $p,t$ with $t \leq p-4$, we have that
    \[
    \frac{2p-3-2t}{5(2p-3)(p-t)} \geq \frac{1}{8p}.
    \]
\end{lemma}
\begin{proof}
    For fixed $p$, consider the function $f(t)=(2p-3-2t)/[5(2p-3)(p-t)]$. Its derivative is $f'(t) = -3/[5(2p-3)(p-t)^2]$, which is strictly negative since $p \geq t+4 \geq 4$. So it suffices to prove the lemma for $t=p-4$, where we have
    \[
    \frac{2p-3-2t}{5(2p-3)(p-t)}=\frac{2p-3-2(p-4)}{20(2p-3)} = \frac{5}{20(2p-3)} = \frac{1}{8p-12}\geq \frac{1}{8p}.\qedhere
    \]
\end{proof}
With these preliminaries, we are ready to prove \cref{lem:removal-lemma}, and thus \cref{thm:removal-above-threshold}.

\begin{proof}[Proof of \cref{lem:removal-lemma}]
 The result is vacuously true if there is no $K_p$ in $\Gamma$, so we may assume that there is at least one copy of $K_p$ in $\Gamma$. 
 Let $v_1,\dotsc,v_p \in V(\Gamma)$ span a copy of $K_p$ in $\Gamma$. For $i \in [p]$, let $V_i = N(v_i)$ denote the neighborhood of $v_i$. By the minimum degree assumption on $\Gamma$, we have that $\ab{V_i} \geq (\frac{2p-5}{2p-3} + \beta)m$ for each $i$. 
 
 For every $S \subseteq [p]$, we define $V_S \coloneqq \bigcap_{i \in S} V_i$.
 We will prove the following claim by induction.

	\begin{claim} 
		For each integer $0 \leq t \leq p-3$, there exists a set $S_t \subseteq [p]$ of size $\ab{S_t} = p-t$ such that $V_{S_t}$ contains at least $(m/(8p))^{t}$ copies of $K_t$.
	\end{claim}
	\begin{proof}[Proof of claim]
		The base case $t=0$ is trivial, since we simply take $S_0 = [p]$.

 		Inductively, suppose we have found such a set $S_t$, for $t \leq p-4$. Let $Q$ be a copy of $K_t$ in $V_{S_t}$.  Since every vertex in $Q$ has degree at least $(\frac{2p-5}{2p-3} + \beta)m$, there are at most $(\frac{2}{2p-3} - \beta)m < \frac{2}{2p-3}m$ vertices not adjacent to any given vertex in $Q$. Thus, if we let $C$ be the common neighborhood of the vertices of $Q$, we find that $C$ has order at least
		\[
			s \coloneqq m - t \left(\frac{2}{2p-3} m \right) = \frac{2p - 3 -2t}{2p-3} m.
		\]
		Note that for any vertex $v \in V(\Gamma)$ and for any set $C' \subseteq V(\Gamma)$ with $\ab{C'}=s$, the number of edges between $v$ and $C'$ is at least
		\[
			\left(\frac{2p-5}{2p-3} + \beta \right)m - (m-s) > \frac{2p-5 - 2t}{2p-3}m = \frac{2p-5-2t}{2p-3-2t}s.
		\]
  We now define an auxiliary bipartite graph $B_t$, as follows. Its first part has $p-t$ vertices, which we identify with $S_t$. Its second part consists of an arbitrary subset $C' \subseteq C$ with $\ab{C'}=s$. We declare that a pair $(i,v) \in S_t \times C'$ is an edge of $B_t$ if $v \in V_i$.
  
  By the computation above, each vertex in the first part of $B_t$ has degree at least $\frac{2p-5-2t}{2p-3-2t}s$.
		The first part of $B_t$ has $\ab{S_t}=p-t$ vertices. Hence, the average degree in the second part of $B_t$ is at least $(p-t)\frac{2p-5-2t}{2p-3-2t}$, which is at least $p-t-2+ \frac25$, by \cref{lem:calc-inequality} applied to $x=p-t$ and using the fact that $t \leq p-4$, which implies that $x\geq 4$. Now, Markov's inequality implies that at least $s/5$ vertices in $C'$ have degree at least $p-t-1$ in the graph $B_t$. 

  Recall that the argument above worked for any fixed choice of a clique $Q$ in $V_{S_t}$, and by the induction assumption, there are at least $(m/(8p))^t$ choices for such a $Q$.
  Therefore, there are at least $(m/(8p))^t \cdot (s/5)$ choices of a clique $Q\subseteq V_{S_t}$, and a common neighbor of $Q$ that lies in at least $p-t-1$ of the sets $V_i$ for $i \in S_t$. Averaging over all subsets of $S_t$ of order $p-t-1$, we conclude that for at least $(m/(8p))^t \cdot s/ (5 (p-t))$ of these choices, the same $(p-t-1)$-subset of $S_t$ is used. We let $S_{t+1}$ be this subset. By definition, the number of copies of $K_{t+1}$ in $V_{S_{t+1}}$ is at least $(m/(8p))^t \cdot s/ (5 (p-t))$.
  To complete the proof of the claim, we note that by \cref{lem:ineq-two},
  \[
  \left(\frac m{8p}\right)^t \cdot \frac{s}{5(p-t)} = \frac{m^{t+1}}{(8p)^t} \cdot \frac{2p-3-2t}{5(2p-3)(p-t)} \geq \frac{m^{t+1}}{(8p)^t} \cdot \frac{1}{8p} = \left(\frac{m}{8p}\right)^{t+1}.\qedhere
  \]
	\end{proof}

	To conclude, we use essentially the same argument for $t=p-3$, except that now we need to keep track of the parameter $\beta$. Let $S = S_{p-3}$ be the set given by the claim for $t=p-3$. As before, fix a copy $Q$ of $K_{p-3}$ in $V_S$. Let $B$ be the auxiliary bipartite graph constructed as above:  its first part has three vertices, labeled by the elements of $S$, and its second part consists of $s$ arbitrary common neighbors of $V(Q)$, where
	\[
		s \coloneqq m - (p-3) \left( \frac{2}{2p-3} - \beta \right) m = \left(\frac{3}{2p-3} + (p-3) \beta\right)m.
	\]
	By the same argument as above, each vertex in the first part of $B$ has degree at least
	\begin{align*}
		\left(\frac{2p-5}{2p-3} + \beta \right)m - (m-s) &= \left( \frac{1}{2p-3} + (p-2) \beta \right) s \\
		&= \frac{\frac{1}{2p-3}+(p-2)\beta}{\frac{3}{2p-3}+(p-3) \beta}s\\
		&=\frac{1+(p-2)(2p-3)\beta}{3+(p-3)(2p-3) \beta} s.
	\end{align*}
    Note that the minimum degree of $\Gamma$ is at most $m$, which implies that $\beta \leq 2/(2p+3)$. This implies that
    \[
    \frac{1+(p-2)(2p-3)\beta}{3+(p-3)(2p-3) \beta} = \frac 13 + \frac{(2p-3)^2}{3\beta(p-3)(2p-3)+9}\beta \geq \frac 13 + \frac{2p-3}{3}\beta.
    \]
    Therefore, the average degree in the second part of $B$ is at least $1+(2p-3)\beta$. Markov's inequality again tells us that at least $\frac {2p-3}2 \beta s$ vertices in this part have at least two neighbors in the first part. Hence, there are at least $(m/(8p))^{p-3}\cdot \frac {2p-3}2 \beta s$ choices of a $K_{p-3}$ in $V_S$ and a vertex in its common neighborhood which lies in at least two of the three sets $V_i$ for $i \in S$. Averaging over the three possible $2$-subsets of $S$, there is some $\{i,j\} \subset S$ such that $V_i \cap V_j$ contains at least $(m/(8p))^{p-3}\cdot \frac {2p-3}6 \beta s$ copies of $K_{p-2}$. We now compute that
	\[
  \left(\frac{m}{8p}\right)^{p-3} \cdot \frac{2p-3}{6}\beta s = \frac{m^{p-2}}{(8p)^{p-3}} \cdot \frac{2p-3}{6}\left(\frac{3}{2p-3} + (p-3) \beta\right)\beta \geq \frac{m^{p-2}}{(8p)^{p-3}} \cdot \frac \beta 2 \geq \beta \left(\frac{m}{8p}\right)^{p-2}.
	\]
	Recalling the definitions of $V_i,V_j$, we see that the edge $\{v_i,v_j\}$ in our original $K_p$ lies in at least $\beta (m/(8p))^{p-2}$ copies of $K_p$, as claimed.
\end{proof}

Finally, we are ready to prove \cref{thm:removal-above-threshold}. As in \cite{2105.09194}, the theorem follows quickly from \cref{lem:removal-lemma}.
\begin{proof}[Proof of \cref{thm:removal-above-threshold}]
    Let $\delta = (10p)^{-2p} \beta \lambda$. Let $\Gamma$ be an $m$-vertex graph with minimum degree at least $(\frac{2p-5}{2p-3}+\beta)m$ and with at most $\delta m^p$ copies of $K_p$.

	Let $E^*$ denote the set of edges in $\Gamma$ which lie in at least $ \beta (m/(8p))^{p-2}$ copies of $K_p$. Each copy of $K_p$ in $\Gamma$ contains at most $\binom p2$ edges from $E^*$, and hence the number of $K_p$ in $\Gamma$ is at least $\binom p2 ^{-1} \beta (m/(8p))^{p-2} \ab{E^*}$. On the other hand, we assume that $\Gamma$ has at most $\delta m^p$ copies of $K_p$.  Combining these two bounds, we see that
	\[
		\ab{E^*} \leq \frac{\delta m^p}{\binom p2 ^{-1} \beta (m/(8p))^{p-2}} = \frac{\delta \binom p2 (8p)^{p-2}}{\beta} m^2 = \frac{\binom p2 (8p)^{p-2} }{(10p)^{2p}} \lambda m^2 \leq \lambda m^2.
	\]
	Additionally, by \cref{lem:removal-lemma}, we know that every $K_p$ in $\Gamma$ contains at least one edge from $E^*$. Said differently, $G \setminus E^*$ is a $K_p$-free graph. As $\ab{E^*} \leq\lambda m^2$ edges, this completes the proof.
\end{proof}

\section*{Acknowledgments} 
We would like to thank Vladimir Nikiforov for bringing \cite[Theorem 9]{BoNi} to our attention, which can be used to give an alternative proof of \cref{thm:stability}. We would also like to thank the anonymous referees for carefully reading the paper and making a number of helpful suggestions.

\bibliographystyle{amsplain}


\begin{aicauthors}
\begin{authorinfo}[jfox]
  Jacob Fox\\
  Stanford University\\
  Stanford, CA 94305\\
  jacobfox@stanford.edu \\
  \url{https://stanford.edu/~jacobfox/}
\end{authorinfo}
\begin{authorinfo}[xhe]
  Xiaoyu He\\
  Princeton University\\
  Princeton, NJ 08540 \\
  xiaoyuh@princeton.edu \\
  \url{https://alkjash.github.io/}
\end{authorinfo}
\begin{authorinfo}[ywig]
  Yuval Wigderson\\
  Tel Aviv University \\
  Tel Aviv 69978, Israel \\
  yuvalwig@tauex.tau.ac.il \\
  \url{http://www.math.tau.ac.il/~yuvalwig/}
\end{authorinfo}
\end{aicauthors}

\end{document}